\documentclass[
  a4paper, 
  reqno, 
  oneside, 
  11pt
]{amsart}

\usepackage[utf8]{inputenc}

\usepackage{vmargin}
\setpapersize{A4}
\setmargins
{2.5cm}		
{1.5cm}		
{16cm}		
{23.42cm}	
{10pt}		
{1cm}		
{0pt}		
{2cm}		

\usepackage{cancel}

\usepackage[dvipsnames]{xcolor} 
\colorlet{cite}{LimeGreen!50!Green}
\usepackage{tikz}  
\usetikzlibrary{arrows,positioning}
\tikzset{ 
  baseline=-2.3pt,
  text height=1.5ex, text depth=0.25ex,
  >=stealth,
  node distance=2cm,
  mid/.style={fill=white,inner sep=2.5pt},
}

\usepackage{lmodern}
\usepackage{tikz-cd}
\usepackage{amsthm, amssymb, amsfonts}

\usepackage{graphicx,caption,subcaption}
\usepackage{fourier}
\usepackage{braket}  
\usepackage{microtype}

\usepackage[%
  bookmarks=true,			
  unicode=true,			
  pdftitle={Lagrangian Skeleta, collars and duality},		%
  pdfauthor={Elizabeth Gasparim}{Francisco Rubilar}{Bruno Suzuki},	%
  pdfkeywords={Duality}{adjoint orbit}{cotangent bundle},	
  colorlinks=true,		
  linkcolor=Blue,			
  citecolor=cite,		
  filecolor=magenta,		
  urlcolor=RoyalBlue			
]{hyperref}				
\usepackage{cleveref}

\usepackage{fourier}
\usepackage{tikz}
\usetikzlibrary{matrix}

\usepackage{tikz-cd}

\newtheoremstyle{mydef}
{}		
{}		
{}		
{}		
{}	
{}		
{ }		
{\thmname{#1}\thmnumber{ #2}\thmnote{ #3}}	

\newtheoremstyle{newdef}
{4pt}
{}
{}
{0pt}
{\bfseries}
{.}
{ }
{\thmname{#1}\thmnumber{ #2}\textnormal{\thmnote{ (#3)}}}

\theoremstyle{plain}
\numberwithin{equation}{section}%
\newtheorem{theorem}[equation]{Theorem} 

\newtheorem{corollary}[equation]{Corollary}
\newtheorem{proposition}[equation]{Proposition}
\newtheorem{lemma}[equation]{Lemma}

\DeclareMathOperator{\Spec}{Spec}
\DeclareMathOperator{\SU}{SU}
\DeclareMathOperator{\Tot}{Tot}
\DeclareMathOperator{\Pic}{Pic}
\DeclareMathOperator{\Elm}{Elm}

\DeclareMathOperator{\Ext}{Ext}

\newcommand{\ce}{\mathrel{\mathop:}=}

\author{E. Ballico, E. Gasparim, F. Rubilar, B. Suzuki}
\title{Lagrangian Skeleta, Collars and Duality}
\begin{document}
	\begin{abstract}
		We present a geometric realization of the  duality between skeleta in $T^*\mathbb P^n$  and  collars of local surfaces. 
		Such duality is predicted by combining two auxiliary types of duality: on one side, 
		symplectic duality between $T^*\mathbb P^n$  and a crepant resolution of the $A_n$ singularity;
		on the other side,  toric duality between two types of  isolated quotient  singularities. We give a correspondence
		between Lagrangian submanifolds of a cotangent bundle and vector bundles on a collar, and  describe those birational transformations 
		within the skeleton which are dual  to deformations of vector bundles.
	\end{abstract}

\maketitle
\tableofcontents

\vspace{2cm}

\section{Skeleton to collar duality}

The simplest example of symplectic duality is the one between the cotangent bundle of projective space $T^*\mathbb P^{n-1}$ 
 and the crepant resolution $\widetilde{Y_n}$ of the $A_{n-1}$ singularity obtained as a quotient $Y_n=\mathbb C^2/\mathbb Z_n$ \cite{BLPW, BF}. 
There exists also a duality between $\widetilde{Y_n}$ and the surface $Z_n = \Tot \mathcal O_{\mathbb P^1}(-n)$,
in the sense  that they  are both minimal resolutions of quotient singularities, but their respective singularities  have dual toric fans.
In fact,  the singular surface $\mathcal X_n$ obtained from $Z_n$ by contracting the zero section 
is also a quotient of $\mathbb C^2$ by the cyclic group of $n$ elements, 
but the singularity of $\mathcal X_n$ is of type $\frac1n(1,1)$ whereas the singularity of $Y_n$ is of type $\frac1n(1,n-1).$ 
 Motivated by these two dualities we discuss some features of the resulting duality between $T^*\mathbb P^{n-1}$ and $Z_n$. On one side, 
 we consider $T^*\mathbb P^{n-1}$  together with a complex potential, thus forming a Landau--Ginzburg model, and we study the Lagrangian skeleton of the corresponding Hamiltonian flow; on the other side,
 we describe the behaviour of vector bundles on the surfaces $Z_n$ considered as algebraic varieties. 
    
 In both cases we will  focus our attention on building blocks used for those types of gluing procedures which may be viewed as surgery operations. 
 We will see that vector bundles on what we call the collar of $Z_n$ (see Sec.\thinspace \ref{bundles}) behave similarly to components of the Lagrangian skeleton of $T^*\mathbb P^{n-1}$.  
 
 Denoting by \textbf{bir} a birational transformation applied to a compactified Lagrangian and \textbf{def} a deformation
of the complex structure of a vector bundle (without describing a categorial equivalence) we give a geometric description of a 1-1 correspondence between objects and some essential morphisms. Such a duality is described by the diagram in the following theorem. 
  \begin{theorem} \label{t1}The following diagram commutes:
\begin{figure}[h]\begin{equation}\label{diagr}
\begin{tikzcd}[row sep=10ex, column sep = -1.5ex,
/tikz/column 5/.append style={anchor=base east},
/tikz/column 6/.append style={anchor=base west}
]
  L_j \arrow[d,swap, "\textup{\bf bir}"] 
& \subset T^*\mathbb{P}^{n-1} \arrow[rrr, Leftrightarrow, "\textup{\textbf{dual}}"]   
&\phantom{xxxxx} & \phantom{xxxxx} & \phantom{.}
 \phantom{x...}\mathcal{O}_{Z_{n}^{\circ}}(j)
& \oplus 
&\!\! \mathcal{O}_{Z_{n}^{\circ}}(-j)\phantom{xxx} \\
 L_{j+1}
& \subset T^*\mathbb{P}^{n-1} \arrow[rrr,Leftrightarrow,swap, "\textup{\textbf{dual}}"]
& & & \phantom{.}
\!\!\!\mathcal{O}_{Z_{n}^{\circ}}(j+1)
& \oplus \arrow[u, swap, "\textup{\bf def}"] 
& \mathcal{O}_{Z_{n}^{\circ}}(-j-1) .
\end{tikzcd}
\end{equation}
\caption*{Duality between Lagrangians and vector bundles.}
\end{figure}
\end{theorem}

 The surfaces $Z_n$ have rich moduli spaces of vector bundles, but it 
  is  mainly the restriction of  a vector bundle
  to the collar of $Z_n^\circ$ (see \ref{def:collar}) that plays a role in  this duality. 
 The cotangent bundle is taken with the canonical symplectic structure and Lagrangian skeleta are described in Sec.\thinspace\ref{skel}. 
 Vector bundles on the local surfaces $Z_n$ are building blocks for vector bundles on compact surfaces. In fact,
  a new gluing procedure called {\it grafting} introduced in \cite{GS} explores the local contribution of these building blocks to the top Chern class.
 This grafting procedure was successful in explaining the physics mechanism underlying the phenomenon of instanton decay around a complex line with negative self-intersection, 
 showing that instantons may decay by inflicting curvature to the complex surface that holds them \cite[Sec.\thinspace7]{GS}.
 For a line with self-intersection $-n$, grafting is done via cutting and gluing over a collar $Z_n^\circ$. 
 The set of isomorphisms classes of rank 2 vector vector bundles over such a collar $Z_n^\circ$ presents a behaviour similar to that of the Lagrangian skeleton of the cotangent bundle $T^*\mathbb P^{n-1}$.
 Therefore our construction here offers a geometric interpretation of this particular instance of duality by exploring building blocks of surgery operations on both sides. 
 When considered in families, one Lagrangian in the skeleton 
is taken to the next via a birational transformation (Sec.\thinspace \ref{elm}) whereas a bundle on the collar is taken to another via deformation of the complex structures (Sec.\thinspace \ref{def}). 
 In this sense we may say that when considering objects of this duality, birational transformations  on Lagrangian skeleta occur as dual to deformations of vector bundles.

\section{Lagrangian skeleton of $T^*\mathbb P^n$}\label{skel}

In this section we will calculate  skeleta of  certain Landau--Ginzburg models. By a Landau--Ginzburg
model we mean a complex manifold together with a complex valued function. 

 Let $(M,\omega)$ be a symplectic manifold together with 
a potential $h$. We assume that $h$ is a Morse function.
In the case when $h$ is a real valued function, the stable manifold of a critical point $p$  consists of   all the points in $M$ that are taken 
to $p$ by the  gradient flow of $h$. However,  when $h$ is a complex valued function, even though the stable manifold 
of a point $p$ is still formed by points that flow to $p$, the natural choice is to use  the Hamiltonian flow of $h$
(which can be thought of as the symplectic gradient).
Furthermore, in the cases considered here, the Hamiltonian flow is given by a torus action (as described in Sec.\thinspace \ref{action})
 and the critical points of $h$ 
are the fixed points of such action. 

 Let $L$ be
the union of the stable manifolds  of the Hamiltonian flow of $h$ with respect to the
K\"ahler metric.  Then $L$ is the isotropic
skeleton of $(M, \omega)$.  When $L$ is of middle dimension, it is called the {\bf Lagrangian
skeleton} of $(M, \omega)$. 
In the case of exact symplectic manifolds,  the Lagrangian
skeleton of $M$  is the complement of the locus escaping to infinity under
the natural Liouville flow, see \cite{Ru,STW}.

To describe the Lagrangian skeleton of $T^*\mathbb P^n$, we will use the Hamiltonian torus action.
We start out with  $\mathbb P^n$ described by homogeneous coordinates $[x_0,x_1,\dots ,x_n]$,  covered by 
the usual open charts
$U_i= \{x_i \neq 0\}$. 
We then write  trivializations of the 
cotangent bundle $T^*\mathbb P^n$ taking products $V_i=U_i \times \mathbb C^n$ and  over the $V_0$ chart we write  coordinates as
$V_0=\left\{[1,x_1,\dots, x_n],(y_1,\dots, y_n)\right\}.$
In this chart, we  write the Hamiltonian action of the torus $\mathbb T\ce \mathbb C\setminus \{0\}$ on $T^*\mathbb P^n$ as 
\begin{equation}\label{action2}\mathbb T \cdot V_0=\left\{[1,t^{-1}x_1,\dots, t^{-n} x_n],(ty_1,\dots, t^ny_n)\right\}.\end{equation}
Note that the same action can be written as
$$\mathbb T \cdot V_0=
\left\{[t^n,t^{n-1} x_1,\dots,x_n],(ty_1,\dots, t^ny_n)\right\}.$$

We will now describe the Lagrangian skeleton  corresponding to this Hamiltonian action. 
We start by showing an example, i.e. the case of $T^*\mathbb P^3$ and then we present the general procedure.\\

\paragraph{\bf Example: skeleton of $T^*\mathbb P^3$}

We take $ \mathbb P^3$ with homogeneous coordinates $[x_0,x_1,x_2,x_3]$, and cover it by open sets
$U_i= \{x_i \neq 0\}$ and  charts $\varphi_i\colon U_i \to \mathbb C^3$ given by 
$\varphi_i([x_0,x_1,x_2, x_3])= \left(\frac{x_0}{x_i}, \ldots, \hat{x_i},\ldots, \frac{x_3}{x_i}\right) $.
The transition matrices for the cotangent bundle $T_{ij} \colon\varphi_i (U_i\cap U_j)\rightarrow \mbox{Aut}(\mathbb C^3)$ are
$$T_{01}= \left(\begin{matrix} -x_1^2  &   -x_1x_2  &  -x_1x_3 \cr
                                                         0   &       x_1  & 0 \cr
                                                         0   &      0      &  x_1       \end{matrix}\right)\quad 
    T_{02}=  \left(\begin{matrix}    -x_1x_2    &     -x_2^2 & -x_2x_3\cr
                                                           x_2         &      0  & 0 \cr
                                                             0 & 0 & x_2  \end{matrix}\right)\quad                                          
T_{03}= \left(\begin{matrix}      -x_1x_3   &  -x_2x_3 & -x_3^2   \cr
                                                      x_3  &    0 & 0 \cr
                                                       0 & x_3 & 0   \end{matrix}\right) .$$ 
                                                       
                                                       Consequently,  we can write down a cover for the cotangent bundle as $V_i= U_i \times \mathbb C^3$, and in coordinates
\begin{align*}
V_0&=\left\{[x_0,x_1,x_2,x_3],(y_1,y_2,y_3)\right\},\\
V_1&=\left\{[x_1^{-1},1,x_1^{-1}x_2, x_1^{-1}x_3],(-x_1^2 y_1-x_1x_2y_2-x_1x_3y_3,x_1y_2,x_1y_3)\right\},\\
V_2&=\left\{[x_2^{-1},x_2^{-1}x_1,1,x_2^{-1}x_3],(-x_1x_2y_1-x_2^2y_2-x_2x_3y_3,x_2y_1,x_2y_3)\right\},\\
V_3&=\left\{[x_3^{-1},x_3^{-1}x_1,x_3^{-1}x_2,1],(-x_1x_3y_1-x_2x_3y_2-x_3^2y_3, x_3y_1,x_3y_2)\right\}.
\end{align*}

Now we take the Hamiltonian action of the torus $\mathbb T$ on $T^*\mathbb P^3$ given by
\begin{align*}
\mathbb T \cdot V_0&=\left\{[1,t^{-1}x_1,t^{-2} x_2,t^{-3}x_3],(ty_1,t^2y_2,t^3y_3)\right\}=
\left\{[t^3,t^2x_1,t x_2,x_3],(ty_1,t^2y_2,t^3y_3)\right\},\\
\intertext{and compatibility on the intersections implies that}
\mathbb T \cdot V_1&=\left\{[tx_1^{-1},1,t^{-1}x_1^{-1}x_2,t^{-2} x_1^{-1}x_3],(-t^{-1}(x_1^2 y_1+x_1x_2y_2+x_1x_3y_3),tx_1y_2,t^2x_1y_3)\right\},	\\
\mathbb T \cdot V_2&=\left\{[t^2x_2^{-1},tx_2^{-1}x_1,1,t^{-1}x_2^{-1}x_3],(-t^{-2}(x_1x_2y_1+x_2^2y_2+x_2x_3y_3),t^{-1}x_2y_1,tx_2y_3)\right\},\\
\mathbb T \cdot V_3&=\left\{[t^3x_3^{-1},t^2x_3^{-1}x_1,tx_3^{-1}x_2,1],(-t^{-3}(x_1x_3y_1+x_2x_3y_2+x_3^2y_3),t^{-2} x_3y_1,t^{-1}x_3y_2)\right\}.
\end{align*}

Using these, we calculate the Lagrangians.
\vspace{3mm}

\noindent {\sc Stable manifold of $e_0$} - on $V_0$ we find the points satisfying
$$\lim_{t\rightarrow 0 }[1,t^{-1}x_1,t^{-2} x_2,t^{-3}x_3],(ty_1,t^2y_2,t^3y_3) = [1,0,0,0],(0,0,0)$$
this requires $x_1=x_2=x_3=0$ and we obtain the fibre over the point $[1,0,0,0]$, that is, 
$$\boxed{L_0= T^*_{[1,0,0,0]}\mathbb P^3\sim \mathbb C^3. }$$

\vspace{3mm}
                                                       
 \noindent {\sc Stable manifold of $e_1$} - on $V_1$ we look for the points satisfying

$$
 	\lim_{t\rightarrow 0 }[tx_1^{-1},1,t^{-1}x_1^{-1}x_2,t^{-2} x_1^{-1}x_3],(-t^{-1}(x_1^2 y_1+x_1x_2y_2+x_1x_3y_3),tx_1y_2,t^2x_1y_3) \\= [0,1,0,0],(0,0,0).
$$

 This requires $x_1^{-1}x_2=  x_1^{-1}x_3 = 0 =  x_1^2 y_1+x_1x_2y_2+x_1x_3y_3$, but since $x_1\neq 0$ in this chart, we get 
 $x_2= x_3 = 0 =   y_1$. So, we are left with points having coordinates 
 $[1,x_1,0,0],(0,y_2,y_3)$ on $V_0$ which on $V_1$ become 
 $[x_1^{-1},1,0, 0],(0,x_1y_2,x_1y_3)$. We obtain (after taking the closure, that is, by adding the point $[1,0,0,0],(0,0,0)$)
 the set of points $\{ [1,x_1,0,0],(0,y_2,y_3)\mapsto [x_1^{-1},1,0, 0],(0,x_1y_2,x_1y_3)\}$ so that 
 $$\boxed{L_1=
  \mathcal O_{\mathbb P^1}(-1)\oplus \mathcal O_{\mathbb P^1}(-1).}$$
 
 \vspace{3mm}
 
  \noindent {\sc Stable manifold of $e_2$} - on $V_2$ we look for  the points satisfying
$$
\lim_{t\rightarrow 0 }      [t^2x_2^{-1},tx_2^{-1}x_1,1,t^{-1}x_2^{-1}x_3],(-t^{-2}(x_1x_2y_1+x_2^2y_2+x_2x_3y_3),t^{-1}x_2y_1,tx_2y_3) \\= [0,0,1,0],(0,0,0).	
$$
     
This requires 
$x_2^{-1}x_3=0=x_1x_2y_1+x_2^2y_2+x_2x_3y_3=x_2y_1$ but since $x_2\neq 0$ in this chart, we get
$x_3=0 $ and $x_1x_2y_1+x_2^2y_2=0 = x_2y_1$ and since on this chart $x_2\neq 0$ it follows that  
 $y_1=y_2=0 .$

 We obtain (after taking the closure) the set of points $\{ [1,x_1,x_2,0],(0,0,y_3)\mapsto  [x_2^{-1},x_2^{-1}x_1,1,0],(0,0,x_2y_3)\}$, so 
 $$\boxed{L_2=\mathcal O_{\mathbb P^2}(-1).}$$

\vspace{3mm}
 
  \noindent {\sc Stable manifold of $e_3$} - on $V_3$ we find the points satisfying
$$
\lim_{t\rightarrow 0 }  [t^3x_3^{-1},t^2x_3^{-1}x_1,tx_3^{-1}x_2,1],(-t^{-3}(x_1x_3y_1+x_2x_3y_2+x_3^2y_3),t^{-2} x_3y_1,t^{-1}x_3y_2)= [0,0,0,1],(0,0).	
$$

This requires 
 $x_1x_3y_1+x_2x_3y_2+x_3^2y_3= x_3y_1=x_3y_2= 0$          and since $x_3 \neq 0$ in this chart, we get that $y_1=y_2=y_3=0$.                  
We obtain the set of points $\{ [x_0,x_1,x_2,x_3],(0,0,0)\}$, so 
 $$\boxed{L_3=  \mathbb P^3.}$$

 The generalization of this procedure to higher dimensions  now becomes evident, giving:\\

\paragraph{\bf General case: the skeleton of  $T^*\mathbb P^n$.}
We take $ \mathbb P^n$ with homogeneous coordinates $[x_0,x_1,x_2,\ldots, x_n]$, and cover it by standard open sets
$U_i= \{x_i \neq 0\}$ and charts $\varphi_i\colon U_i \to \mathbb C^n$ given by 
$\varphi_i([x_0,x_1,x_2,\ldots, x_n])= \left(\frac{x_0}{x_i}, \ldots, \hat{x_i},\ldots, \frac{x_n}{x_i}\right) $.
The transition matrices for the cotangent bundle $T_{ij} \colon\varphi_i (U_i\cap U_j)\rightarrow \mbox{Aut}(\mathbb C^n)$ are
$$T_{01}= \left(\begin{matrix} 
	-x_1^2  &   -x_1x_2	&	\cdots	& -x_1x_n	\\
	 0		&       x_1	&	\cdots	& 0			\\
	 \vdots &  \vdots	& \ddots	&	\vdots	\\
	 0		&	0		& \cdots & x_1
 \end{matrix}\right), \qquad 
T_{0n}= \left(\begin{matrix}
	-x_1x_n	& -x_2x_n	& \cdots	& -x_2x_n	& -x_n^2   \\
	x_n 	& 0			& \cdots 	& 0			& 0  \\
	0 		& x_n 		& \cdots	& 0 		& 0 \\
	\vdots	& \vdots 	& \ddots	& 0 		& 0 \\
	0 		& 0 		& \cdots 	& x_n 		& 0
\end{matrix}\right) .
$$ 

$$T_{0j}=\left(\begin{matrix}
-x_jx_1	& -x_jx_2	&\cdots	&-x_j^2	& \cdots	& -x_jx_{n-1}	& -x_jx_n 	\\
x_j 	& 0			&\cdots	& 0		& \cdots	& 0				& 0 		\\
0 		& x_j		&\cdots	& 0		& \cdots	& 0				& 0			\\
\vdots 	& 			&\ddots	&		& 			& 				& 	\vdots	\\
	 	& 			&		& 0		& 			& 				& 			\\
\vdots 	& 			&		&		& \ddots	& 				& 	\vdots	\\
0 		& 0			&\cdots & 0		& \cdots	& x_j			& 0			\\
0 		& 0			&\cdots & 0		& \cdots 	& 0				& x_j
\end{matrix}\right).$$

Consequently,  we can write down a cover for the cotangent bundle as $V_i= U_i \times \mathbb C^n$, and in coordinates
\begin{align*}
V_0&=\left\{[x_0,\ldots,x_n],(y_1,\ldots,y_n)\right\},\\
V_1&=\left\{[x_1^{-1},1,x_1^{-1}x_2, \ldots, x_1^{-1}x_{n-1}, x_1^{-1}x_n],
(-x_1^2y_1 - x_1x_2y_2 - x_1x_3y_3,x_1y_2,\ldots,x_1y_n)\right\},\\
\intertext{\centering \vdots}	
V_j&=\left\{
[x_j^{-1},x_j^{-1}x_1, \ldots 1, \ldots, x_j^{-1}x_n],
(-x_jx_1y_1 - \ldots -x_jx_ny_n, x_jy_2, \ldots, x_jy_n)\right\},\\
\intertext{\centering \vdots}
V_n&=\left\{[x_n^{-1},x_n^{-1}x_1, \ldots, x_n^{-1}x_{n-1},1],
(-x_nx_1y_1 - \ldots - x_n^2y_n, x_ny_2,\ldots, x_ny_n)\right\}.
\end{align*}

Now we take the Hamiltonian action of the torus $\mathbb T$ on $T^*\mathbb P^n$ given by
\begin{align*}
\mathbb T \cdot V_0&=\left\{[1,t^{-1}x_1,t^{-2} x_2,\ldots, t^{-n}x_n],
(ty_1,t^2y_2,\ldots, t^ny_n)\right\}=
\left\{[t^n,t^{n-1}x_1,t^{n-2} x_2, \ldots, x_n],(ty_1,t^2y_2,\ldots, t^ny_n)\right\},\\
\intertext{and compatibility on the intersections implies that}
\mathbb T \cdot V_1&=\left\{
[tx_1^{-1},1,t^{-1}x_1^{-1}x_2,\ldots, t^{n-1} x_1^{-1}x_n],
(-t^{-1}(x_1^2 y_1 + \ldots +x_1x_ny_n),tx_1y_2,\ldots, t^{n-1}x_1y_n)\right\},\\
\intertext{\centering \vdots}
\mathbb T \cdot V_n&=\left\{
[t^nx_n^{-1},t^{n-1}x_n^{-1}x_1,\ldots, tx_n^{-1}x_{n-1},1],
(-t^{-n}(x_1x_ny_1 + \ldots, + x_n^2y_n),t^{-(n-1)}x_ny_1,\ldots, t^{-1}x_ny_n)\right\}.	
\end{align*}

Using these, we calculate the Lagrangians.
\vspace{3mm}

\noindent {\sc Stable manifold of $e_0$} - on $V_0$ we find the points satisfying
$$\lim_{t\rightarrow 0 }[1,t^{-1}x_1,\ldots,t^{-n}x_n],(ty_1,t^2y_2,\ldots, t^ny_n) = [1,0,\ldots,0],(0,\ldots,0)$$
this requires $x_1=x_2=\cdots=x_n=0$ and we obtain the fibre over the point $[1,0,\ldots,0]$, that is, 
$$\boxed{L_0= T^*_{[1,0,\ldots,0]}\mathbb P^n\sim \mathbb C^n. }$$

\vspace{3mm}
                                                       
 \noindent {\sc Stable manifold of $e_1$} - on $V_1$ we find the points satisfying
 \begin{multline*}
 \lim_{t\rightarrow 0 }
 [tx_1^{-1},1,t^{-1}x_1^{-1}x_2,\ldots, t^{-(n-1)} x_1^{-1}x_n],
 (-t^{-1}(x_1^2 y_1 + \ldots +x_1x_ny_n),tx_1y_2,\ldots, t^{n-1}x_1y_n) \\
 = [0,1,0\ldots,0],(0,\ldots,0).
 \end{multline*}
 
  This requires $x_1^{-1}x_2= \cdots =x_1^{-1}x_n = 0 = x_1^2 y_1 + \ldots +x_1x_ny_n$, but since $x_1\neq 0$ in this chart, we get 
 $x_2= \cdots = x_3 = 0 = y_1$. 
 So, we are left with points having coordinates 
 $[1,x_1,0, \ldots,0],(0,y_2,\ldots, y_n)$ on $V_0$ which on $V_1$ become 
 $[x_1^{-1},1,0 \ldots, 0],(0,x_1y_2, \ldots, x_1y_n)$. We obtain (after taking the closure, that is adding the point $[1,0,\ldots,0],(0,\ldots,0)$)
 $$\boxed{L_1=\{ [1,x_1,0,\ldots, 0],(0,y_2,\ldots, y_n) \mapsto
 [x_1^{-1},1,0,\ldots, 0],(0,x_1y_2,\ldots, x_1y_n)\}\sim \mathcal O_{\mathbb P^1}(-1)\oplus \cdots \oplus \mathcal O_{\mathbb P^1}(-1) .}$$
 ($n-1$ summands).
 \vspace{3mm}
 
   \noindent {\sc Other stable manifolds}
 
 Using similar computations, we have that the Lagrangian $L_j$ corresponding to the fixed point $e_j$ is
 
\begin{equation}\label{formskel}\boxed{
L_j = \begin{cases}
	\mathbb C^n&\text{if} \quad  j=0,\\
	\oplus_{i=1}^{n-j}{\mathcal O_{\mathbb P^j}(-1)}&\text{if} \quad 0<j<n,\\
	\mathbb{P}^n&\text{if} \quad j=n.
\end{cases}
}\end{equation}
 
 
\section{Potentials on the cotangent bundle}\label{action}
    
In this section we consider the question:  what choices of  potential $h$ for a Landau--Ginzburg model    $(T^*\mathbb P^n, h)$ 
is compatible with the Hamiltonian action considered in the previous sections, and hence gives rise to the same skeleta?
We obtain the following result. 

\begin{proposition} Consider $(T^*\mathbb P^n, h_c)$ with coordinates 
$[1, x_1,\dots,x_n],(y_1,\dots, y_n)$. Each  potential $$h_c([1, x_1,\dots,x_n],(y_1,\dots, y_n))=\sum_{i=1}^n -2ix_iy_i + c,$$ has
a corresponding Hamiltonian flow that  coincides with the flow obtained by the torus action given in (\ref{action2}), that is 
 $$\mathbb T \cdot V_0=\left\{[1,t^{-1}x_1,\dots, t^{-n} x_n],(ty_1,\dots, t^ny_n)\right\}.$$
\end{proposition}

To prove this, first consider the vector field on $T^\ast\mathbb{P}^n$ corresponding to the Hamiltonian action given  in  coordinates by
$$\mathbb T \cdot V_0=
\left\{[1,t^{-1} x_1,\dots,t^{-n}x_n],(ty_1,\dots, t^ny_n)\right\}.$$
On the image of the $V_0$ chart, the right hand side becomes
\begin{align*}
	\alpha(t)&=(t^{-1} x_1,\dots,t^{-n}x_n,ty_1,\dots, t^ny_n)\\
	\intertext{
		so that the derivative gives}
	\alpha'(t)&= (-t^{-2} x_1,\dots,-nt^{-n-1}x_n,y_1,\dots, nt^{n-1}y_n)\\
	\intertext{	and evaluating at $1$ we get}
	\alpha'(1)&= (-x_1,\dots,-nx_n,y_1,\dots, ny_n).
\end{align*}

From the action of this $1$-parameter subgroup, we have obtained the flow $\alpha'(1)$. 
Now we wish to calculate a potential $h$ corresponding to the vector field $X= \alpha'(1)$.

 Let $\omega$  be the canonical symplectic form on $T^\ast\mathbb{P}^n$, then
$h$ must satisfy, for all vector fields $Z \in \mathfrak{X}(M)$
\[
dh(Z)=\omega(X,Z).
\]
In coordinates this gives
$$
\left(
\begin{matrix}
	\displaystyle\frac{\partial h}{\partial {\bf x}}   \displaystyle\frac{\partial h}{\partial {\bf y}} 
\end{matrix}
\right)
\left(
\begin{matrix}
	{\bf a}  \\ {\bf b} 
\end{matrix}
\right)
=
\sum_{i=1}^n dx_i \wedge dy_i \left( (-x_1, \ldots , -nx_n , y_1, \ldots , ny_n),
(  a_1,\ldots ,a_n , b_1,\ldots, b_n )\right) \\
= -2\sum ix_ib_i+iy_ia_i.
$$
where 
${\bf x}= (x_1, \dots, x_n), {\bf y}= (y_1, \dots, y_n),{\bf a}= (a_1, \dots, a_n), {\bf b}= (b_1, \dots, b_n).$
Comparing the terms multiplying $a_k$ and $b_k$ on each side of the equation, 
for $i = 1, \ldots, n$
 we obtain the differential equations
$$
\frac{\partial h}{\partial x_i}  = -2iy_i, \qquad
\frac{\partial h}{\partial y_i}  = -2ix_i.$$
For $c \in \mathbb{C}$, the solutions are:
\begin{equation}\label{toricpot}
h_c = -2x_1y_1 - \ldots -2nx_ny_n +c= \sum_{i=1}^n -2ix_iy_i + c.
\end{equation}

We thus conclude that any Landau--Ginzburg model of the form $(T^*\mathbb P^n, h_c)$ will give rise to the same 
skeleta described above. This concludes the description of our Landau--Ginzburg models and their skeleta on $T^*\mathbb P^n$ and 
 in the next section we discuss birational maps within each skeleton.

\section{Birational maps within the skeleton}\label{elm}

In this section we present the birational transformations between components of the skeleton that justify the vertical downarrow appearing
on the left hand side of diagram (\ref{diagr}).
As we saw in (\ref{formskel}), the component $L_j$  of the skeleton of $T^*\mathbb P^n$ has the  form 
$$\mathcal O_{\mathbb P^j} (-1) \oplus \mathcal O_{\mathbb P^j} (-1) \oplus \cdots \oplus \mathcal O_{\mathbb P^j} (-1)$$
with $n-j$ factors. Projectivizing we obtain $\mathbb P^j \times \mathbb P^{n-j-1}$. Thus, the component $L_{j+1}$ 
has the form $$\mathcal O_{\mathbb P^{j+1}} (-1) \oplus \mathcal O_{\mathbb P^{j+1}} (-1) \oplus \cdots \oplus \mathcal O_{\mathbb P^{j+1}} (-1)$$
with $n-j-1$ factors. Projectivizing we obtain $\mathbb P^{j+1} \times \mathbb P^{n-j-2}$. 
The projectivizations are birationally equivalent, as we describe next,  and up to tensoring by  $\mathcal O(+1)$, we may choose  a birational map 
 taking $L_j$ to $L_{j+1}$. 

 The birational maps $\mathbb P^n\times \mathbb P^m\dasharrow \mathbb P^{n+m}$ we need here 
 are well known, but we recall one construction for completeness.
We take homogeneous coordinates $y_0,\dots ,y_n$ on $\mathbb P^n$ and $z_0,\dots ,z_m$ on $\mathbb P^m$. Set $r:= (n+1)(m+1)-1$ and take homogeneous coordinates $u_{ij}$, $0\le i \le n$, $0\le j\le m$, of $\mathbb P^r$. 
Let $\nu: \mathbb P^n\times \mathbb P^m \to \mathbb P^r$ be the Segre embedding of 
$\mathbb P^n\times \mathbb P^m$ into $\mathbb P^r$ given by the equations $u_{ij} =y_iz_j$.
The birational map (not a morphism, since  there is no birational morphism between these two varieties)
is induced by a linear projection $\ell_M: \mathbb P^r\setminus M\to \mathbb P^{n+m}$, where $M$ is an $(r-n-m-1)$-dimensional linear subspace
 whose equations are coordinates $u_{ij}=0$ for some $i, j$ and the $n+m+1$ homogeneous coordinates
of $\mathbb P^{n+m+1}$ are the ones used to describe $M$. Recall that linear projections in suitable coordinates are just rational maps which forget some of the coordinates.

We start by considering the simplest example, that is, $n=m=1$ and hence $r=3$. Therefore $M$ is a point, say $([0:1],[0:1])$. Take  for $\mathbb P^3$ homogeneous coordinates 
$$x_0= y_0z_0, \quad x_1= y_1z_0, \quad x_2 =y_0z_1, \quad x_3= y_1z_1$$ with 
$M = [0:0:0:1]$ and use $x_0,x_1,x_2$ for coordinates of $\mathbb P^2$.
 
The next step is to consider $n=2$, $m=1$ and hence $r=5$. Take $\mathbb P^5$ with
homogeneous coordinates $$x_0=y_0z_0,\quad x_1 =y_0z_1,\quad x_2=y_1z_0,\quad x_3=y_1z_1,\quad x_4 = y_2z_0,\quad y_5=y_2z_1.$$
Then $M$ is a line contained in the first ruling of $\mathbb P^2\times \mathbb P^1$
so it has the form $L\times \{p\}$ where $L\subset \mathbb P^2$ is a line, 
and $p\in \mathbb P^1$ is a point. If we take $L = \{y_0=0\}$ and $p = [0:1]$ 
we get the equations $x_0=x_1=x_2=x_4=0$ and  the coordinates of $\mathbb P^3$ should be  $x_0,x_1,x_2,x_4$.
We blow-up $L\times \{p\}\subset \mathbb P^2\times \mathbb P^1$ and 
then we contract the strict transform of $\mathbb P^2\times \{p\}$ and $L\times \mathbb P^1$. 
So, the birational map is clear. 

Now, to take one Lagrangian to the next one, we argue in generality. 
Suppose we have 2 quasi-projective varieties $X$, $X'$, with Zariski open 
subsets $U\subseteq X$, $V\subseteq X'$, $U\ne \emptyset$, such that there exists an
 isomorphism $$s\colon U \to V.$$
If for a fixed quasi-projective variety $Y$, 
we need two proper birational morphisms
$$u_1\colon Y \to X\qquad \text{and} \qquad u_2: Y\to X'$$ compatible with $s$, then we have a single
 choice: take first the graph $$W:= \{(x,s(x)\}_{x\in U} \subset U\cup V,$$
then take the closure $T$ of $W$ in $X\times X'$.
Then $T$ has the two morphisms 
$$v_1\colon T\to X\qquad \textup{ and }\qquad v_2\colon T\to X'$$ and any other $(Y,u_1,u_2)$ must be obtained 
by composing $(T,v_1,v_2)$
with a morphism $f: Y\to T$, in such a way that we obtain $$u_1:= f\circ v_1\qquad \textup{ and }\qquad u_2:= f\circ v_2.$$

The argument in this section shows that we have a birational transformation taking $L_j$ to $L_{j+1}$, thus justifying 
the  vertical downarrow {\bf bir} appearing in Thm.\thinspace\ref{t1}
     we now proceed to discuss the other side of the duality in focus here, namely  singularities and vector bundles on their resolutions.
                                                                                                 
\section{Duality for multiplicity $n$ singularities}\label{sing}

We describe a  duality between vector bundles on 2 distinct minimal resolutions of toric singularities of multiplicity $n$,
 which are both quotients of $\mathbb C^2$ by the cyclic group of $n$ elements $\mathbb Z / n \mathbb Z$, 
and whose toric cones are dual, they are:  
$$\mathcal X_n\ce \frac1n(1,1)\quad \text{and} \quad  \mathcal X^\vee_n\ce\frac1n(1,n-1).$$
These singularities are obtained by the following actions:
\[
\begin{pmatrix}
\rho & 0 \\
0 & \rho
\end{pmatrix} \quad \text{for  } \mathcal X_n \quad 
\text{and} \quad 
\begin{pmatrix}
\rho & 0 \\
0 & \rho^{-1}
\end{pmatrix} \quad \text{for } \mathcal X^\vee_n,\]
where $\rho$ is a primitive $n$-th root of unity, that is, $\rho= e^{\frac{2\pi i}{n}}$.
In general the singularity $\frac1n(1,a)$ is obtained from the action $(x,y) \mapsto (\rho x , \rho^a y)$.

A resolution of singularities  $\widetilde{X} \rightarrow X$ is called {\bf minimal}  if $\widetilde{X} \rightarrow X'  \rightarrow X$ 
with $X'$ smooth imply $\widetilde{X} \simeq X'$. 
Let ${ Z}_n$ and $\widetilde{Y_{n}}$ denote the minimal toric resolutions of $\mathcal X_n$ and $\mathcal X^\vee_n$, respectively,
depicted in Fig. \ref{fans1} and \ref{fans2}.
We observe that, in particular we have $Z_2 \simeq \widetilde{Y_{2}}$,
but $Z_n \not \simeq  \widetilde{Y_{n}}$ for $n\neq 2$.

 The surface $ {Z}_n = \Tot \mathcal O_{\mathbb P^1} (-n)$ contains a single rational curve with self-intersection $-n$, 
whereas  $\mathcal X^\vee_n$ contains an isolated $A_{n-1}$-singularity and  $\widetilde{Y_{n}}$ contains a chain of $n-1$  curves $E_i \simeq \mathbb P^1$
for $1 \leq i \leq n-1$  whose intersection matrix $(E_i \cdot E_j)$ coincides with the negative of the Cartan matrix of the simple 
Lie algebra $\mathfrak{sl}_n (\mathbb C)$ of type $A_{n-1}$.

Note that
the Dynkin diagram\, 
\raisebox{.17em}{%
\begin{tikzpicture}[x=1.5em,y=1.4em]
\foreach \pair in {{(0,0)},{(1,0)},{(2,0)},{(4,0)}}{%
\draw[line width=.5pt, fill=black] \pair circle(0.3ex);
}
\node[font=\scriptsize,outer sep=0,inner sep=1pt] (dots) at (3,0) {\,$\cdots$};
\draw[line width=.6pt] (0,0) -- (dots);
\draw[line width=.6pt] (dots) -- (4,0);
\end{tikzpicture}%
\, }
of type $A_{n-1}$ is precisely the  graph dual to the system of curves $E_i$ in the resolution of $\widetilde{Y_{n}}$.


The surfaces $\mathcal X_n$ and $\mathcal X^\vee_n$ are   toric varieties  having   fans formed by a single cone,
calling $\sigma_{\mathcal X_n}$ and $\sigma_{\mathcal X^\vee_n}$ their respective fans, 
we have that $\sigma_{\mathcal X_n}$ is dual to $\sigma_{\mathcal X^\vee_n}$.
In particular, for the case of $n=2$ we also have that $\sigma_{Z_2} \simeq \sigma_{\widetilde{Y_{2}}}$ is self-dual.

We now describe the coordinate rings of the singularities.
We have  $\mathcal X_n = \Spec A$, where
\begin{equation}\label{xis}
A = H^0 (Z_n, \mathcal O) \simeq \mathbb C [x_0, \dotsc, x_n] / (x_i x_{j+1} - x_{i+1} x_j)_{0 \leq i < j < n}.
\end{equation}

Given that $\mathcal X_n \simeq \mathbb C^2 / \Gamma$, where $\Gamma$ is the group generated by
$\left(
\begin{smallmatrix}
\rho & 0 \\
0 & \rho
\end{smallmatrix}
\right)$
for $\rho$ a primitive $n$-th root of unity, we have $\Gamma \simeq \mathbb Z / n \mathbb Z$, with $j \in \mathbb Z / n \mathbb Z$ corresponding to $\left( \begin{smallmatrix} \rho^j & 0 \\ 0 & \rho^j \end{smallmatrix} \right)$.

Functions on the quotient $\mathbb C^2 / \Gamma$ are given by those functions on $\mathbb C^2$ which are invariant under $\Gamma$.

The algebra of functions on $\mathbb C^2$ is $\mathbb C [a, b]$ and $\Gamma$ acts by multiplication by $\rho$ on both $a$ and $b$.
 We thus have that $a^i b^j = (\rho a)^i (\rho b)^j = \rho^{i+j} a^i b^j$ if and only if $\rho^{i+j} = 1$, 
 i.e.\ if and only if $i + j$ is a multiple of $n$. One sees that $\mathbb C [a, b]^\Gamma$ (functions on $\mathbb C^2$ invariant under $\Gamma$) are generated by
$$
a^n, a^{n-1} b, \dotsc, a b^{n-1}, b^n.
$$
Now one can check that the invariants are
\[
\mathbb C [a, b]^\Gamma = \mathbb C [a^n, a^{n-1} b, \dotsc, a b^{n-1}, b^n] \simeq A
\]
with the resolution mapping 
$$a^i b^{n-i} \mapsto x_i \quad \text{for} \quad 0 \leq i \leq n,$$ so that $\mathbb C^2 / \Gamma \simeq \mathcal X_n$.
This map looks quite similar to the Veronese embedding. In fact, $\mathcal X_n$ is the so-called 
affine cone over the {\it Veronese curve} (or {\it rational normal curve}) of degree $n$, i.e.\ $\mathcal X_n \simeq \mathbb C^2 / \Gamma$ 
is the affine cone over the image of the Veronese embedding $\mathbb P^1 \to \mathbb P^n$ given by 
$[a : b] \mapsto [a^n : a^{n-1} b : \dotsb : a b^{n-1} : b^n]$.

The duality between $\mathcal X_n$ and $\mathcal X_n^\vee$ is made clear by their toric fans. Just observe that 
each fan consists of a single cone, and the vectors forming the fan of $\mathcal X_n$ are perpendicular to those of
the fan of $\mathcal X_n^\vee$ as depicted in Fig. \ref{fans1} and \ref{fans2}.

\begin{figure}
\begin{align*}
\begin{tikzpicture}[x=2em,y=2em]
\node at (.75,-.75) {{\it fan of $\mathcal X_n$}};
\draw[line width=.6pt, line cap=round] (-1,3) -- (-1.11,3.33);
\draw[line width=.6pt, line cap=round] (1,0) -- (3.5,0);
\draw[-stealth,line width=.6pt, line cap=round] (0,0) -- (-.99,2.97);
\draw[-stealth,line width=.6pt, line cap=round] (0,0) -- (.97,0);
\foreach \coordinate in {{(-1,3)}, {(0,0)}, {(0,1)}, {(0,2)}, {(0,3)}, {(1,0)}, {(1,1)}, {(1,2)}, {(1,3)}, {(2,0)}, {(2,1)}, {(2,2)}, {(3,0)}, {(3,1)}}{%
\draw[line width=.6pt, fill=black] \coordinate circle (.16ex);
}
\begin{scope}[shift={(6,0)}]
\node at (.75,-.75) {{\it fan of $\widetilde{\mathcal X}_n=Z_n$}};
\draw[line width=.6pt, line cap=round] (-1,3) -- (-1.11,3.33);
\draw[line width=.6pt, line cap=round] (1,0) -- (3.5,0);
\draw[-stealth,line width=.6pt, line cap=round] (0,0) -- (-.99,2.97);
\draw[-stealth,line width=.6pt, line cap=round] (0,0) -- (.97,0);
\draw[-stealth,line width=.6pt, line cap=round] (0,0) -- (0,.97);
\draw[line width=.6pt, line cap=round] (0,1) -- (0,3.5);
\foreach \coordinate in {{(-1,3)}, {(0,0)}, {(0,1)}, {(0,2)}, {(0,3)}, {(1,0)}, {(1,1)}, {(1,2)}, {(1,3)}, {(2,0)}, {(2,1)}, {(2,2)}, {(3,0)}, {(3,1)}}{%
\draw[line width=.6pt, fill=black] \coordinate circle (.16ex);
}
\end{scope}
\end{tikzpicture}
\end{align*}
\caption{$ Z_n$ as toric resolution of $\mathcal X_n$}
\label{fans1}
\end{figure}
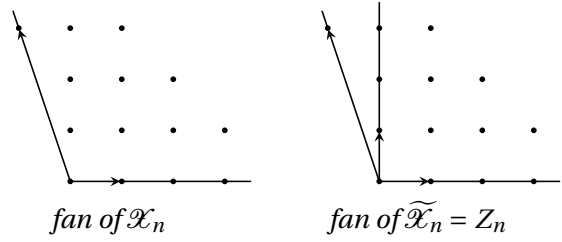

\begin{figure}
\begin{align*}
\begin{tikzpicture}[x=2em,y=2em]
\node at (.75,-.75) {{\it fan of $\mathcal X^\vee_n=Y_n$}};
\draw[line width=.6pt, color=white] (-1,3) -- (-1.11,3.33);
\draw[line width=.6pt, line cap=round] (0,1) -- (0,3.5);
\draw[line width=.6pt, line cap=round] (3,1) -- (3.33,1.11);
\draw[-stealth,line width=.6pt, line cap=round] (0,0) -- (2.97,.99);
\draw[-stealth,line width=.6pt, line cap=round] (0,0) -- (0,.97);
\foreach \coordinate in {{(0,0)}, {(0,1)}, {(0,2)}, {(0,3)}, {(1,1)}, {(1,2)}, {(1,3)}, {(2,1)}, {(2,2)}, {(3,1)}}{%
\draw[line width=.6pt, fill=black] \coordinate circle (.16ex);
}
\begin{scope}[shift={(6,0)}]
\node at (.75,-.75) {{\it fan of\, $\widetilde{Y_n}$}};
\draw[line width=.6pt, line cap=round] (1,1) -- (2.47,2.47);
\draw[line width=.6pt, line cap=round] (2,1) -- (3.12,1.56);
\draw[line width=.6pt, line cap=round] (3,1) -- (3.33,1.11);
\draw[-stealth,line width=.6pt, line cap=round] (0,0) -- (2.97,.99);
\draw[-stealth,line width=.6pt, line cap=round] (0,0) -- (0,.96);
\draw[-stealth,line width=.6pt, line cap=round] (0,0) -- (.97,.97);
\draw[-stealth,line width=.6pt, line cap=round] (0,0) -- (1.96,.98);
\draw[line width=.6pt, line cap=round] (0,1) -- (0,3.5);
\draw[line width=.6pt, color=white] (1,0) -- (3.5,0);
\foreach \coordinate in {{(0,0)}, {(0,1)}, {(0,2)}, {(0,3)}, {(1,1)}, {(1,2)}, {(1,3)}, {(2,1)}, {(2,2)}, {(3,1)}}{%
\draw[line width=.6pt, fill=black] \coordinate circle (.16ex);
}
\end{scope}
\end{tikzpicture}
\end{align*}
\caption{$\widetilde{Y_{n}}$ as toric resolution of $\mathcal X^\vee_n=Y_n$}
\label{fans2}
\end{figure}
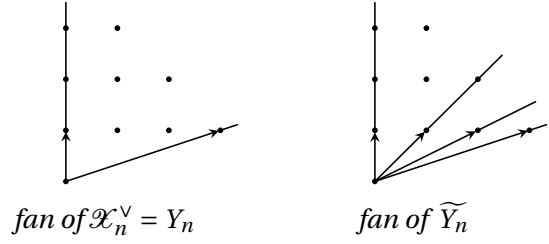

  
\section{Vector bundles on local surfaces}   
 \label{bundles}   
 
 We now describe vector bundles on $Z_n$, the resolution of the isolated singularity $\mathcal X_n$.
 The surface $Z_n$ is the local model of the neighborhood of a rational line $\ell$ with self-intersection $-n$ 
 in a complex surface $X$. 
 Thus, vector bundles on $Z_n$ model vector bundles around 
such a line $\ell$ in $X$. The case $n=1$ occurs when blowing-up a smooth point, and was explored in \cite{G}.
 
 Recently, a new  complex surgery operation on vector bundles over $Z_n$, named {\it grafting}, was introduced 
 in the context of mathematical physics (see \cite{GS}). It provided an original 
 explanation for the phenomenon of instanton decay in terms of curvature of the underlying space.
  Here we explore  the geometric features of this grafting procedure. 
 When considered from the point of view of grafting, bundles on $Z_n$ occur as building blocks of vector bundles on surfaces, in a sense somewhat analogue (and dual) to the use of the Lagrangian skeleton for  building a symplectic manifold.

 Let $E$ be a vector bundle on a compact complex surface $X$ which contains a $-n$ line. 
 Let $F= E\vert_N$ be the restriction of $E$ to an open neighborhood $N$ of $\ell$ in the analytic topology. Grafting is obtained by replacing $F$ by another vector bundle $F'$, which is then glued to $E\vert_{X\setminus N}$. 
 Note that after grafting the top Chern class of $E$ will in general change, but not the first one. 
 Therefore, this surgery procedure is not obtained by an elementary transformation. 
 The gluing itself is done over $N\setminus \ell$ which is   identified with the complement of the zero section in $Z_n$ called the collar defined below;
 such a gluing is possible because vector bundles on $Z_n$ are completely determined by their restriction to a finite formal neighborhood of $\ell$,
 see \cite{BGK}. 
 We now describe explicit local data on $Z_n$ used to classify vector bundles on them. 
 These vector bundles restricted to the collars will give rise to the dual objects to the components of the skeleta described above.

\label{Zkdef}

For each integer $n$, we have the surface
$
Z_n = \Tot (\mathcal{O}_{\mathbb{P}^1}(-n)).
$
The complex manifold structure  can be described by gluing the open sets 
$$U = \mathbb{C}[z,u]  \quad \mbox{and} \quad  V = \mathbb{C}[\xi,v]$$ 
by the relation
\begin{equation}\label{Zcanonical}
(\xi, v) = (z^{-1}, z^n u)
\end{equation}
whenever $z$ and $\xi$ are not equal to 0.
We call (\ref{Zcanonical}) the {\bf canonical coordinates} for $Z_n$.

Using canonical coordinates, the contraction $Z_n \ce \Tot \mathcal O_{\mathbb P^1} (-n)  \to \mathcal X_n$ sends $z^i u \mapsto x_i$,
where $x_i$ are the coordinates of $\mathcal X_n$ as described in (\ref{xis}).

 Let $E$ be a rank $r$ holomorphic vector bundle on $Z_n$. 
 The restriction of $E$ to the zero section $\ell \simeq \mathbb P^1$ is a rank $r$ bundle on $\mathbb P^1$, which by Grothendieck's lemma splits as a direct sum of line bundles.
 Thus, $E \vert_\ell \simeq \mathcal O_{\mathbb P^1} (j_1) \oplus \dotsb \oplus \mathcal O_{\mathbb P^1} (j_r)$. 
 Following \cite{ballico}, we call $(j_1, \dotsc, j_r)$ the  splitting type of $E$. 
 When $E$ is a rank $2$ bundle with first Chern class $0$, then the splitting type is $(j,-j)$ for some $j \geq 0$ and we say for short that $E$ has \textbf{splitting type} $j$.

 There are many rank 2 vector bundles on $Z_n$. 
 For each fixed splitting type, they can be obtained as a quotient of $\Ext^1(\mathcal O(j), \mathcal O(-j))$. 
Considering isomorphism classes of vector bundles modulo holomorphic equivalence,  moduli spaces were  obtained as follows.

\begin{proposition}
\cite[Thm.\ 4.11]{BGK}
The moduli space of irreducible $\SU (2)$ instantons on $Z_n$ with charge (and splitting type) $j$ is a quasi-projective variety of dimension $2j-n-2$. 
\end{proposition}

An equivalent formulation in terms of vector bundles is:

\begin{corollary}
The moduli space of (stable) rank 2 bundles on $Z_n$ with vanishing first Chern class and local 
second Chern class $j$ is a quasi-projective variety of dimension $2j-n-2$. 
\end{corollary}

Even though vector bundles on $Z_n$ are many, their restrictions to the collars have very simple behaviour, 
as we now shall demonstrate.

We denote by $\ell$ the $\mathbb P^1$ contained in  $Z_n$  corresponding to 
the zero section of the corresponding vector bundles, and we set 
\begin{equation}
	Z_n^o \ce Z_n \setminus \ell.
\label{def:collar}
\end{equation}
  
We call  $Z_n^o$ the {\bf collar} of $\ell$ in $Z_n$. 
Using the canonical coordinates for $Z_n$  we obtain {\bf canonical coordinates} for the collar by 
 setting
\[
Z_n^o = U^o \cup V^o, 
\]
with the complex manifold structure obtained by gluing the open sets 
$$U^o = \mathbb{C}\times \mathbb C-\{0\}\simeq \mathbb C[z,u, u^{-1}]  \quad \mbox{and} 
\quad  V^o = \mathbb{C}\times \mathbb C-\{0\}\simeq \mathbb C[\xi,v, v^{-1}]$$ 
by the relation
$$
(\xi, v) = (z^{-1}, z^n u).
$$

\begin{lemma}\label{homotopy}
The homotopy type of $Z_n^o$ is that of an $S^1$-bundle over $S^2$, and $\pi_1(Z_n^o)= \mathbb Z/n\mathbb Z$.
\end{lemma}

\begin{proof}
Let $D= \{z, |z|\leq 1\}$ be the unit disc in $\mathbb C$, denoted $D^+$ when oriented positively,
and $D^-$ when oriented negatively.
The homotopy type of $Z_n^o$  is then that of 
$$U^o \sim U^+=D^+ \times S^1= [z,u=e^{i \theta}]  \quad \mbox{and} 
\quad  V^o \sim U^-=  D^- \times S^1= [\xi,v=e^{i\phi}]$$ 
with identification in $U^+ \cap V^-$ given by 
$$(\xi= e^{i\alpha}, v=e^{i\phi})= (z^{-1}=e^{-i\alpha}, v= z^nu= e^{i(\theta+n\alpha)}).$$
The result of the identification is an $S^1$-bundle over $ S^2=D^+\cup D^-$, with the $S^1 $ fibers 
identified via the gluing map $z^n $ which has degree $n \in \pi_1(SO(2))=\mathbb Z$ since $SO(2)\simeq S^1$.
\end{proof}

 Let $\iota \colon Z_n^o \rightarrow Z_n$ denote the inclusion, and set
$$L_n(j)\ce \iota^*\mathcal O_{Z_n} (j).$$

\begin{proposition}
For each $n$, the group of  all isomorphism classes  of line bundles
$ \{L_n(j), j \in \mathbb Z\}$ is cyclic of order $n$, hence  $\mathbb{Z}/n\mathbb{Z}$.
\end{proposition}

\begin{proof}
Note that $\Pic(Z_n) = \mathbb{Z}$. Each line bundle over $Z_n$ with first Chern class $j$ 
is isomorphic to $\mathcal O_{Z_n}(j)$ and therefore can be represented by a transition matrix $(z^{-j})$.
Since    in canonical coordinates   we have that $u^{-1}\neq 0$ and $v \neq 0$ 
on the collar $Z_n^o$, we may change coordinates  as follows 
\[
(z^{-j}) \simeq (v)(z^{-j})(u^{-1}) = (z^nu \cdot z^{-j} \cdot u^{-1}) = (z^{-j+n}),
\]
i.e., over $Z_n^o$, the bundles $L_n(j)$ and $L_n(j-n)$ (defined by $(z^{-j})$ and $(z^{-j+n})$ respectively) are isomorphic. 
Moreover, if $j_1 \equiv j_2 \mod  n$, then $L_n(j_1)$ and $L_n(j_2)$ are isomorphic. The proof 
that the cases $1,2,...,n-1$ are not pairwise isomorphic is included in the proof of Prop.\thinspace \ref{elms}.
\end{proof}

The following is a slightly rephrased version of \cite[Prop.\thinspace4.1]{GKM}.

\begin{proposition}\label{elms}Let $E_1$ and $E_2$ be 
rank 2 bundles over $Z_n$ with vanishing first Chern classes 
 and splitting types $j_1$ and $j_2$, respectively. There exists an isomorphism $E_1\vert_{Z_n^o} \simeq E_2\vert_{Z_n^o}$ if  and only if $j_1 \equiv j_2 \, \textrm{mod} \, n$. In particular, $E_1$ 
is trivial over $Z_n$ if and only if $j_1 \equiv 0 \, \textrm{mod} \, n$.
\end{proposition}

\begin{proof}
We first claim that the bundle $\mathcal{O}_\ell(-n)$ is trivial on $Z_n^o$. In fact, if $u=0$ is the equation of $\ell$, then $s(z,u)=u$ determines a section of $\mathcal{O}_\ell(-n)$ that does not vanish on $Z_n^o$.

If a bundle $E$ over $Z_n$ has splitting type $j$, then by definition, $E\vert_\ell \cong \mathcal{O}_\ell(-j) \oplus \mathcal{O}_\ell(j)$. 
So there is a surjection $\rho \colon E\vert_\ell \to \mathcal{O}_\ell(j)$, and a corresponding  elementary transformation, 
resulting in a vector bundle 
 $E' = \Elm_{\mathcal{O}_\ell(j)}(E)$ which splits over 
$\ell$ as $\mathcal{O}_\ell(-n) \oplus \mathcal{O}_\ell(j+n)$, see \cite[Sec.\thinspace 3]{BGK}. 
Therefore we can use the surjection $\rho \colon E'\vert_\ell \to \mathcal{O}_\ell(j+n)$ to perform a second elementary transformation, and we obtain the bundle $E'' = \Elm_{\mathcal{O}_\ell(j+n)}(E')$, which splits over $\ell$ as $\mathcal{O}_\ell(-j) \oplus \mathcal{O}_\ell(j+2n)$ and has first Chern class $2n$. 
Tensoring by $\mathcal{O}_\ell(-n)$ we get back to a $\mathfrak{sl}_2(\mathbb{C})$-bundle with splitting type $j+n$. Hence, the transformation
\[
\Phi(E) = \otimes \mathcal{O}(-n) \circ \Elm_{\mathcal{O}_\ell(j+n)} \circ \Elm_{\mathcal{O}_\ell(j)}(E)
\]
increases the splitting type by $n$ while keeping the isomorphism type of $E$ over $Z_n^o$. So we need only to analyze bundles with splitting type $j < n$. 

If $j=0$, the bundle is globally trivial on $Z_n$. If $j \neq 0$, then $E\vert_{Z_n^o}$ induces a non-zero element on the fundamental group $\pi_1 (Z_n ^o) = \mathbb{Z}/ n \mathbb{Z}$.

By Lem.\thinspace \ref{homotopy} the collar $Z_n^o$ has the homotopy type of an $S^1$-bundle over $S^2$ and $\pi_1 (Z_n ^o) = \mathbb{Z}/ n \mathbb{Z}$.
Therefore $H_1(Z_n^o, \mathbb Z) = \mathbb{Z}/ n\mathbb{Z}$ and by Poincar\'e duality $H^2(Z_n^o, \mathbb Z) = \mathbb{Z}/ n\mathbb{Z}$.
The exponential sheaf sequence 
$$0 \rightarrow \mathbb Z \rightarrow \mathcal O \rightarrow \mathcal O^* \rightarrow 0$$ induces 
the first Chern class map 
$$H^1(Z_n^o, \mathcal O^*) \rightarrow H^2(Z_n^o, \mathbb Z) = \mathbb{Z}/ n \mathbb{Z},$$
and $$L_n(j) \mapsto j \mod n.$$
\end{proof}

In this section we have described vector bundles on $Z_n$, their moduli, and behaviour on collars. 
We will  see next that each splitting type is connected to the lower ones by deformations.

\section{Deformations}\label{def}     
      In this section we justify the vertical upwards arrow appearing in diagram (\ref{diagr}).
      We start with a vector bundle $E$ with splitting type $(j, -j)$ on $Z_n$, so that $E$ may be written as an  extension
     \begin{equation}\label{extclass} 0 \rightarrow \mathcal O(-j) \rightarrow E \rightarrow \mathcal O(j) \rightarrow 0 .\end{equation}
 Alternatively, we may also choose to write $E$ as an extension of $\mathcal O(j+s)$ by $\mathcal O(-j-s)$ for any $s>0$.
 To see this, just observe that there exist inclusions
$$H^1(\mathcal O(-2j)) = \Ext^1(\mathcal O(j), \mathcal O(-j)) \stackrel{\iota}{\hookrightarrow} \Ext^1(\mathcal O(j+s), \mathcal O(-j-s))=H^1(\mathcal O(-2j -2s)).$$

Let  $p$ be the extension class corresponding to representing the bundle $E$ by the exact sequence (\ref{extclass}).
Next, fixing  an injection $\iota$,  consider the family   $t\cdot \iota (p)$ of 
extensions of $\mathcal O(j+1)$ by $\mathcal O(-j-1)$. For such a family, when
$t=0$ we obtain $\mathcal O(j+1) \oplus \mathcal O(-j-1)$ but when $t=1$ we obtain $E$. 

Now, using induction on $j$, we conclude that every bundle on 
 $Z_n$ occurs as a deformation of another bundle with  splitting type as high as desired. In particular, such  behaviour
 of lowering the splitting type via deformations  is also observed over
 the collars, justifying the vertical uparrow {\bf def} appearing in Thm.\thinspace \ref{t1}.
  We now combine this vertical uparrow with the vertical downarrow {\bf bir} described in Sec.\thinspace \ref{elm}.
 There is a  1-1 correspondence between elements of the skeleton and splitting types on the collar.
Given that  this correspondence is obtained via  a combination of 2 dualities, we call it  a  duality transformation.
We denote it by a horizontal double arrow:
 $$L_j \Longleftrightarrow \mathcal{O}_{Z_{n}^{\circ}}(j)
\oplus  \mathcal{O}_{Z_{n}^{\circ}}(-j).$$
  
  Collecting horizontal and vertical arrows together, we obtain the commutative diagram claimed in Thm.\thinspace \ref{t1}.
  \begin{figure}[h]
  $$
\begin{tikzcd}[row sep=10ex, column sep = -1.5ex,
/tikz/column 5/.append style={anchor=base east},
/tikz/column 6/.append style={anchor=base west}
]
  L_j \arrow[d,swap, "\textup{\bf bir}"] 
& \subset T^*\mathbb{P}^{n-1} \arrow[rrr, Leftrightarrow, "\textup{\textbf{dual}}"]   
&\phantom{xxxxx} & \phantom{xxxxx} & \phantom{.}
 \phantom{x...}\mathcal{O}_{Z_{n}^{\circ}}(j)
& \oplus 
&\!\! \mathcal{O}_{Z_{n}^{\circ}}(-j)\phantom{xxx} \\
 L_{j+1}
& \subset T^*\mathbb{P}^{n-1} \arrow[rrr,Leftrightarrow,swap, "\textup{\textbf{dual}}"]
& & & \phantom{.}
\!\!\!\mathcal{O}_{Z_{n}^{\circ}}(j+1)
& \oplus \arrow[u, swap, "\textup{\bf def}"] 
& \mathcal{O}_{Z_{n}^{\circ}}(-j-1) .
\end{tikzcd}
$$
\end{figure}

In conclusion, we have given an explicit geometric description of a duality between Lagrangians in the skeleta of cotangent bundles 
 and 
vector bundles on collars.
The symplectic side of the duality studies the components  of the Lagrangian 
 skeleta of cotangent bundles over $n$-dimensional projective spaces. 
 The complex algebraic side  considers
 only $2$-dimensional complex varieties. 
 These 2 are rather different types of objects.
 So, a priori this duality was not at all evident, but was abstractly predicted by a combination of 2 other types of duality.

In future  work, we intend to pursue a generalization of this 
type of duality to the realm of Calabi--Yau threefolds, investigating what symplectic manifolds and Lagrangians 
are dual to vector bundles on local Calabi--Yau varieties  and what operations occur as dual to 
deformations of vector bundles, see \cite{GR, GKRS}. The latter promises to be a challenging question, given the existence of 
infinite dimensional families of deformations in the case of  $3$-dimensional varieties, see \cite{BGS}.

 \section*{Acknowledgements}
\noindent
E. Ballico is a member of  MIUR and GNSAGA of INdAM (Italy).
E. Gasparim was partially supported by  the Vicerrector\'ia de Investigaci\'on y Desarrollo Tecnol\'ogico de la  Universidad Cat\'olica del Norte (Chile) 
and by the Department of Mathematics of the University of Trento. 
 B. Suzuki was supported by the ANID-FAPESP cooperation 2019/13204-0.
F. Rubilar acknowledges support from  Beca Doctorado Nacional -- Folio 21170589. 
Gasparim, Rubilar and Suzuki acknowledge  support of MathAmSud GS\&MS 21-MATH-06 - 2021/2022.
We are grateful to I. Cheltsov for inviting us to contribute to this volume.


\begin{thebibliography}{WW}

\bibitem[B]{ballico}
E. Ballico,
\textit{Rank $2$ vector bundles in a neighborhood of an exceptional curve of a smooth surface},
Rocky Mountain  J. Math.  {\bf 29} n.4 (1999) 1185--1193.

  \bibitem[BGK]{BGK} 
E. Ballico, E. Gasparim,  T. K\"oppe,
\textit{Vector bundles near negative curves: moduli and local Euler characteristic},
Comm. Algebra {\bf 37} (2009) 2688--2713.
  
  \bibitem[BGS]{BGS} 
E. Ballico, E. Gasparim, B. Suzuki, \textit{Infinite dimensional families of Calabi-Yau threefolds and moduli of vector bundles}, J. Pure Appl. Algebra \textbf{225} n.4 (2021)
106554. 
 
 
\bibitem[BF]{BF} M. Brion, B. Fu, \textit{Symplectic Resolutions for Conical Symplectic Varieties}, Int. Math. Res. Not. IMRN \textbf{2015} n.12 (2014) 4335--4343.
 
 
  
  
    
\bibitem[BLPW]{BLPW} T. Braden, A. Licata, N. Proudfoot, B. Webster, \textit{Quantizations of conical symplectic resolutions II: category $\mathcal{O}$ and symplectic duality}, arXiv:1407.0964. 
 

  

\bibitem[G]{G} E. Gasparim, \textit{The Atiyah--Jones conjecture for rational surfaces},
      Advances Math. {\bf 218} (2008) 1027--1050.

\bibitem[GKM]{GKM} E. Gasparim, T. Koppe, P. Majumdar, {\it Local Holomorphic Euler Characteristic and Instanton Decay}, 
Pure  App. Math. Q. {\bf 4} n.2 (2008) 363--382. 
  
  
   
\bibitem[GKRS]{GKRS} 
E. Gasparim, T. K\"oppe, F. Rubilar, B. Suzuki, \textit{Deformations of noncompact Calabi--Yau threefolds}, Rev. Colombiana de Matem\'aticas 
{\bf 52} n.1 (2018) 41--57. 
  
  
\bibitem[GR]{GR} 
E. Gasparim, R. Rubilar, \textit{Deformations of Noncompact Calabi--Yau Manifolds, 
Families and Diamonds}, Contemp. Math. {\bf 766}  (2021) 117--132.
  
  

\bibitem[GS]{GS} E. Gasparim, B. Suzuki \textit{Curvature grafted by instantons}, Indian J. Phys. {\bf 95} n.8 (2021) 1631--1638. 




\bibitem[Ru]{Ru} W.-D. Ruan,  \textit{The Fukaya category of symplectic
neighborhood of a non-Hausdorff manifold}, arXiv:0602119. 

\bibitem[STW]{STW} V. Shende, D. Treumann,  H. Williams, \textit{On the combinatorics of exact Lagrangian surfaces}, arXiv:1603.07449.

\end{thebibliography}
\end{document}